\documentclass[12pt,final]{amsart}
\usepackage[margin=1in,includeheadfoot]{geometry}
\usepackage{amsmath}
\usepackage{slashed}
\usepackage{graphicx}

\usepackage[margin=1in,includeheadfoot]{geometry}
\usepackage{amsmath}
\usepackage{slashed}
\usepackage{graphicx}

\usepackage[margin=1in,includeheadfoot]{geometry}
\usepackage{amsmath}
\usepackage{slashed}
\usepackage{graphicx}
\usepackage{mathrsfs}
\usepackage{txfonts}
\usepackage{amssymb,dsfont}
\usepackage{enumerate}
\usepackage{slashed}
\usepackage{fancyhdr}
\usepackage{aliascnt}
\usepackage{pdfsync}
\usepackage{hyperref}

\usepackage{color}

\newtheorem{thm}{Theorem}[section]

\newtheorem{rem}[thm]{Remark}
\newtheorem{lem}[thm]{Lemma}
\newtheorem{prop}[thm]{Proposition}
\newtheorem{defn}{Definition}[section]

\numberwithin{equation}{section}

 \newcommand{\al}{\alpha}

 \newcommand{\de}{\delta}

 \newcommand{\Si}{\Sigma}
 \newcommand{\si}{\sigma}

 \newcommand{\ga}{\gamma}

 \renewcommand{\th}{\theta}

 \newcommand{\Real}{\mathbb{R}}
 
 \newcommand{\Complex}{\mathbb{C}}

 \def\<{\left\langle} \def\>{\right\rangle}
 \def\({\left(} \def\){\right)}
 \newcommand{\n}{\nabla}
 \newcommand{\p}{\partial}

\numberwithin{equation}{section}

\renewcommand{\t}[1]{\tilde{#1}}

\title[Harmonic maps from degenerating Riemann surfaces with boundary]{Harmonic maps with free boundary from degenerating bordered Riemann surfaces}
\author{Lei Liu}
\address{Albert-Ludwigs-Universit\"{a}t Freiburg, Mathematisches Institut, Ernst-Zermelo-Strasse 1, D-79104, Freiburg im Breisgau, Germany}%
\email{lei.liu@math.uni-freiburg.de}%
\author{Chong Song}
\address{School of Mathematical Sciences, %
  Xiamen University, %
Xiamen 361005, P. R. China}%
\email{songchong@xmu.edu.cn}
\author{Miaomiao Zhu}
\address{School of Mathematical Sciences, %
  Shanghai Jiao Tong University, %
Shanghai 200240, P. R. China}%
\email{mizhu@sjtu.edu.cn}%

\date{\today}

\begin{document}
\allowdisplaybreaks[4]
\begin{abstract}
We study the blow-up analysis and qualitative behavior for a sequence of harmonic maps with free boundary from degenerating bordered Riemann surfaces with uniformly bounded energy. With the help of Pohozaev type constants associated to harmonic maps defined on degenerating collars, including vertical boundary collars and horizontal boundary collars, we establish a generalized energy identity.
\end{abstract}

\maketitle

\section{Introduction}\label{sec:intro}

\

Let $(\Si, c)$ be a compact bordered Riemann surface with smooth boundary $\partial \Si$ and with complex structure $c$. Let $(N, g)$ be a compact Riemannian manifold. Let $K\subset N$ be a $k-$dimensional closed submanifold with $1\leq k\leq n$ and denote
\[ C(K)=\left \{ \  u\in C^2(\Sigma, N);   \  u(\p\Si)\subset K    \   \right \}.\]
Let $h$ be a Riemannian metric on $\Si$ which is compatible with the complex structure $c$. A critical point of the energy functional \begin{equation}\label{functional}
E(u)=\int_\Sigma|\n u|^2dvol_h
\end{equation}
over the space $C(K)$ is called a \emph{harmonic map with free boundary} on $K$.

By Nash's embedding theorem, we embed $(N,h)$ isometrically into some Euclidean space $\mathbb{R}^N$. Then the Euler-Lagrange equation of the functional \eqref{functional} is
\[
\Delta_h u=A(u)(\nabla u,\nabla u),
\]
where $A$ is the second fundamental form of $N\subset \mathbb{R}^N$ and $\Delta_h$ is the Laplace-Beltrami operator on $M$ which is defined as follows
$$\Delta_h:=-\frac{1}{\sqrt{h}}\frac{\partial}{\partial x^\beta}\left(\sqrt{h}h^{\alpha\beta}\frac{\partial}{\partial x^\alpha}\right).$$
Moreover, for $1\leq k\leq n-1$, $u$ has free boundary $u(\partial \Si)$ on $K$, namely, the following holds
\begin{align}\label{FB}
u(x)\in K,\quad du(x)(\overrightarrow{n})\perp T_{u(x)}K, \quad \forall \ x\in\partial \Si,
\end{align}
where $\overrightarrow{n}$ is the outward unit normal vector on $\partial \Si$ and $\perp$ means orthogonal. For $k=n$, $u$ satisfies a homogeneous Neumann boundary condition on $K$, that is
\begin{align}
u(x)\in K,\quad du(x)(\overrightarrow{n})=0, \quad \forall \ x\in\partial \Si.
\end{align}

The tension field $\tau(u)$ of a map $u$ is defined by
\begin{align}\label{HM}
&\tau(u)=-\Delta_h u+A(u)(\nabla u,\nabla u).
\end{align}
Thus, if $u: (\Si, h) \to (N,g)$ is a harmonic map with free boundary on $K$ then it satisfies the Euler-Lagrangian equation
\begin{equation}\label{eq}
  \left\{\begin{aligned}
    &\tau(u)=0&\text{~in~}\Si;\\
    &du(\vec{n}) \ \bot  \ T_uK&\text{~on~}\p\Si,
  \end{aligned}\right.
\end{equation}
where $\tau(u)$ is the tension field of $u$. It is well known that both the energy functional \eqref{functional} and the harmonic map system \eqref{eq} are conformally invariant, hence they are independent of the choice of the metric $h$ in the same conformal class $c=[h]$.

Now suppose $(\Si_n, c_n)$ is a sequence of bordered Riemann surface of the same topological type $(g,k)$ such that
\begin{equation}\label{g-bar}
\bar{g}=2g+k-1>1,
\end{equation}
here $g$ is the genus of $\Sigma_n$, $k$ is the number of components of $\partial \Sigma_n$ and $\bar{g}$ is the genus of the complex double $\Si_n^c$ of $\Si_n$ (see Section \ref{sec:bordered-Riemann-surf}).
Let $u_n:\Si_n\to N$ be a sequence of harmonic maps with free boundary on the submanifold $K\subset N$, whose energy is uniformly bounded
$$E(u_n)\le \Lambda<\infty.$$

When the domain surface is fixed, namely $\Si_n=\Si$, it was shown in \cite{jost-Liu-Zhu} that the sequence $u_n$ converges to a limit harmonic map modulo finitely many bubbles, i.e. harmonic spheres or harmonic disks with free boundary on $K$. Moreover, the energy identity and the no neck property hold, extending the classical case of a closed domain surface developed in \cite{SU, Jost1991, parker1996bubble} etc. and the case of harmonic functions \cite{LP}.

If we allow the conformal structure $c_n$ on the domain surface to vary. Then, the conformal structure might degenerate and in such a case, the limit surface is a nodal bordered Riemann surface. When the Riemann surface $(\Si_n, c_n)$ is closed, by the uniformization theorem, there is a hyperbolic metric $h_n$ in this conformal class $c_n$. It is well-known that the degeneration of conformal structures $(\Si_n, h_n, c_n)$ is obtained by shrinking finitely many simple closed geodesics $\ga_n$ (for simplicity of notation, we assume there is only one such geodesic) to a point, which is known as the node. Moreover, there is a collar area near $\ga_n$ which is conformal to a long standard cylinder $$P_n=[-T_n,T_n]\times S^1.$$
In \cite{Zhu}, the asymptotic behaviour of the maps in the limit process was successfully characterized by associating to each $P_n$ a quantity, called Pohozaev type constant, defined by the Hopf differential of the map. More precisely, by integrating the Hopf differential on a slice of $P_n$, one gets a constant\:
\[
\alpha_n:=\int_{\{t\}\times S^1}\left(|\partial_t u_n|^2-|\partial_\theta u_n|^2 \right)d\theta, \quad \text{~for~} P_n,
\]
then the limit energy and length of $u_n$ on $P_n$ can be expressed in terms of $\al_n$ and $T_n$. This result can be considered as a 2nd order extension of Gromov's compactness for J-holomorphic curves \cite{Gr}.

\

In the present paper, we shall investigate the situation that the underlying Riemann surface $(\Si_n, c_n)$ is bordered and in particular, the degeneration of $(\Si_n, c_n)$ occurs at the boundary. In fact, for each $(\Si_n, c_n)$, we can construct a double cover $(\tilde{\Si}_n, \tilde{c}_n)$, which is closed. Then the degeneration of $(\tilde{\Si}_n, \tilde{c}_n)$ is exactly the same as in the closed case described above. We will see that, if the degeneration happens at the boundary, then there is a collar area which is conformal to the following two types of half cylinders:

\

\noindent{\bf Type I:}  a vertical half cylinder
$$Q_n^+=[-T_n,T_n]\times [0,\pi],$$
\noindent{\bf Type II:} a horizontal half cylinder
$$P_n^+=[0,T_n]\times S^1,$$

\noindent to which we can associate a similar quantity defined in terms of the Hopf differential of the map.

More precisely, in this paper, we shall define the Pohozaev type constant associated to $u_n$ over the above two types of half cylinders (see Lemma \ref{lem:01}) by
\[
\alpha_n:=\left\{\begin{aligned}
&\int_{\{t\}\times [0,\pi]}\left(|\partial_t u_n|^2-|\partial_\theta u_n|^2 \right )d\theta,   &\text{~for~}  Q_n^+;   \\
&\int_{\{t\}\times S^1}\left(|\partial_t u_n|^2-|\partial_\theta u_n|^2 \right)d\theta,  &\text{~for~} P_n^+ .
\end{aligned}\right.
\]

\

Combined with the case of interior degeneration studied in \cite{Zhu}, now we state the following:

\begin{thm}\label{thm:main}
Let $u_n:(\Sigma_n,h_n,c_n)\to (N,g)$ be a sequence of harmonic maps with free boundary and with uniformly bounded energy $E(u_n,\Sigma_n)\leq\Lambda<\infty$, where $(\Sigma_n,h_n,c_n)$ is a sequence of compact hyperbolic Riemann surfaces with smooth boundary $\partial\Sigma_n$ and of genus $\bar{g}>1$ (see \ref{g-bar}), degenerating to a hyperbolic Riemann surface $(\Sigma,h,c)$ by the following three ways:
\begin{itemize}
\item {\bf interior degeneration:} \
collapsing finitely many pairwise disjoint  interior simple closed geodesics $\left \{\gamma_{1n}^{j},j=1,...,p_1 \right \}\subset\Sigma_n\setminus \partial\Sigma_n$;\\
\item  {\bf Type I boundary degeneration:} \
collapsing finitely many pairwise disjoint simple geodesics \\
$\left\{\gamma_{2n}^{j},j=1,...,p_2\right\}$, where each geodesic $\gamma_{2n}^{j}$, connects two points on the boundary $\partial\Sigma_n$;\\
\item  {\bf Type II boundary degeneration:} \
collapsing finitely many pairwise disjoint boundary simple closed geodesics $\left\{\gamma_{3n}^{j},j=1,...,p_3\right\}\subset \partial\Sigma_n$.
\end{itemize}
For each $j=1,2...,p_i$, $i=1,2,3$, the geodesic $\gamma_{in}^{j}$ degenerates into a pair of punctures $(\varepsilon_i^{j,1},\varepsilon_i^{j,2})$. Here we use the convention that $p_i=0$ for some $i=1,2,3$ means that such type of degeneration does not occur.

Denote the $h_n$-length of $\gamma_{in}^{j}$ by $l_{in}^{j}$.  Then, after passing to a subsequence, there exist finitely many blow-up points $\{x_1,...,x_I\}$ which are away from the punctures $(\varepsilon_i^{j,1},\varepsilon_i^{j,2}),i=1,2,3;j=1,...,p_i$, and finitely many harmonic maps:
\begin{itemize}
  \item $u:(\overline{\Sigma},\overline{c})\to N$ with free boundary $u(\partial\overline{\Sigma})$ on $K$, where $(\overline{\Sigma},\overline{c})$ is the normalization of $(\Sigma, c)$;
  \item $\widetilde{v}^l_m:S^2\to N$, $l=1,2,...,\widetilde{L}_m$, near the blow-up point $x_m$, $m=1,2,...,I$;
  \item $\widetilde{w}^k_m:D_1(0)\to N$ with free boundary $\widetilde{w}^k_m(\partial D_1(0))$ on $K$, $k=1,2,...,\widetilde{K}_m$, near the blow-up point $x_m$, $m=1,2,...,I$;
  \item $v^{j,l}_i:S^2\to N$, $l=1,2,...,L^j_i$, near the puncture $(\varepsilon_i^{j,1},\varepsilon_i^{j,2}),i=1,2,3;j=1,...,p_i$ ;
 \item $w^{j,k}_i:D_1(0)\to N$ with free boundary $w^{j,k}_i(\partial D_1(0))$ on $K$, $k=1,2,...,K^j_i$, near the puncture $(\varepsilon_i^{j,1},\varepsilon_i^{j,2}),i=2,3;j=1,...,p_i$ ;
\end{itemize}
such that $u_n$ converges to $u$ in $C^\infty_{loc}(\Sigma\setminus \{x_1,...,x_I\})$ and the following generalized energy identity holds
\begin{align*}
\lim_{n\to\infty}E(u_n)=E(u)&+\sum_{m=1}^{I}\sum_{l=1}^{\widetilde{L}_m}E(\widetilde{v}^l_m) +\sum_{m=1}^{I}\sum_{k=1}^{\widetilde{K}_m}E(\widetilde{w}^k_m)\\
&+\sum_{i=1}^{3}\sum_{j=1}^{p_i}\sum_{l=1}^{L_i^j}E(v^{j,l}_i) +\sum_{i=2}^{3}\sum_{j=1}^{p_i}\sum_{k=1}^{K_i^j}E(w^{j,k}_i) \\
&+\sum_{j=1}^{p_1}\lim_{n\to\infty}\frac{2\pi^2}{l^j_{1n}}\alpha^j_{1n} +\sum_{i=2}^3\sum_{j=1}^{p_i}\lim_{n\to\infty}\frac{\pi^2}{l^j_{in}}\alpha^j_{in},
\end{align*}
where $\alpha^j_{in} $ is the Pohozev type constant defined on the collar area near  $\gamma^j_{in} $, $i=1, 2,3$,  $j=1,2,...,p_i$.
\end{thm}

\begin{rem}
Theorem~\ref{thm:main} can be compared with the compactness results of (twisted) holomorphic curves with Lagrangian boundary conditions (cf.~\cite{thesis-Liu, Xu}), which can be viewed as a 1st order analog of the case of harmonic maps with free boundary investigated in this paper.
\end{rem}

The rest of the paper is organized as follows. In Section \ref{sec:bordered-Riemann-surf}, we describe the geometric structure of nodal bordered Riemann surface and the two types of boundary nodes. In section \ref{sec:proof-of-thm}, we develop some analytic properties of harmonic maps from two types of long half cylinders and then prove our main Theorem \ref{thm:main}.

\

\section{Geometric structure of nodal bordered Riemann surface }\label{sec:bordered-Riemann-surf}

\

In this section, we collect some facts about nodal bordered Riemann surfaces and the two types of boundary nodes as well as the corresponding two types of boundary collar regions. For more details, we refer to e.g. \cite{thesis-Liu,Katz-Liu}.

\subsection{Bordered Riemann surface}

A smooth bordered Riemann surface $\Sigma$ is said to be of type $(g, k)$ if $\Sigma$ is topologically a sphere attached with $g$
handles and $k$ disks removed. It is topologically equivalent to a compact surface of genus $g$ with $k$ punctures.

For any bordered Riemann surface $\Sigma$, there exists a compact double cover surface $\Si^c$ with an anti-holomorphic involution map $\si: \Si^c\to \Si^c$ such that $\Si=\Si^c/\si$, and $\p\Si$ is just the fixed point set of $\si$. Given a type $(g, k)$ smooth bordered Riemann surface $\Sigma$, the genus of its complex double $\Si^c$ is $\tilde{g} = 2g + k - 1$. In fact, we have

\begin{thm}
Let $\Si$ be a bordered Riemann surface. There exists a double cover
$\pi:\Si^c\to \Si$ of $\Si$ by a compact Riemann surface $\Si^c$ and an anti-holomorphic involution
$\si:\Si^c\to \Si^c$ such that $\pi\circ\si=\pi$. There is a holomorphic embedding $\iota:\Si\to \Si^c$
such that $\pi\circ\iota$ is the identity map. The triple $(\Si^c,\pi,\sigma)$ is unique up to isomorphism.
\end{thm}
\begin{proof}
See \cite{book-AG}[Theorem 1.6.1] and also \cite{book-SS}.
\end{proof}

Actually, the double cover $\Si^c$ is decided only by the differentiable structure of $\Si$, which can be constructed as follows. First, take a maximal atlas $\mathcal{U}=\{U_i \ | \ i\in I\}$ of $\Si$, such that any compatible atlas of $\Si$ is contained in $\mathcal{U}$. Then the coordinate charts in $\mathcal{U}$ can be divided into two classes according to the orientation. Now let $S=\amalg_{i\in I}U_i$ be the disjoint union. For each point in $U_i\cap U_j$, we identify the corresponding points in $U_i$ and $U_j$ if their transition map is orientation-preserving. In this way, we obtain the orientable covering $\Si^o$.

If $\p\Si\neq\emptyset$, then for all $p\in \p\Si$, we identify the corresponding two points in $\Si^o$ to get the complex double $\Si^c$, which is closed. The covering $\pi:\Si^o\to \Si$ induces a
mapping $\pi: \Si^c \to \Si$ which is a ramified double covering of $\Si$. It is a local
homeomorphism at all points $p\in \Si^c$ for which $\pi(p)\notin \p\Si$. At points lying
over the boundary of $\Si$, the projection is a folding similar to the mapping
$x + iy\to x + i|y|$ at the real axis.

If $\Si$ is orientable, then $\Si^o$ is a trivial double covering of $\Si$, which is disconnected and consists of two copies of $\Si$ with opposite orientations. Then $\Si^c$ is obtained by simply gluing the boundary of the two components.

\

\subsection{Nodes}

A nodal bordered Riemann surface is of type $(g, k)$ if it is
a degeneration of a smooth bordered Riemann surface of the same type.

Let $\Sigma$ be a (smooth) bordered Riemann surface of type $(g, k)$, then the following are
equivalent:
\begin{itemize}
  \item $\Sigma$  is stable, i.e., $Aut(\Sigma)$ is finite
  \item Its complex double $\Sigma^c$ is stable
  \item The genus $\tilde{g} = 2g + k - 1$ of $\Sigma^c$ is greater than one
  \item The Euler characteristic $\chi(\Sigma)=2-2g-k$ is negative
\end{itemize}

 There are three types of nodes for a nodal bordered Riemann surface. Let $(z, w)$ be the coordinate on $\mathbb{C}^2$
 and $\si(z, w) = (\bar{z}, \bar{w})$ be the complex
conjugation. A node on a bordered Riemann surface is a singularity which is locally isomorphic to one of
the following:
\begin{enumerate}
  \item Interior node: $(0,0)\in \{zw=0\}$
  \item Type I boundary node: $(0,0)\in \{z^2-w^2=0\}/\si$
  \item Type II boundary node: $(0,0)\in \{z^2+w^2=0\}/\si$
\end{enumerate}

\

\subsection{Interior nodes}

Let $U$ be a neighborhood of $(0,0)\in \Complex^2$ and $\Sigma=\{(z,w)\in U|zw=0\}$. Then $\Sigma$ has two components $\Sigma_1=\{(z,w)\in U|z=0\}$ and $\Sigma_2=\{(z,w)\in U|w=0\}$, which are attached at the nodal point $\Sigma_1\cap \Sigma_2=\{(0,0)\}$. Note that the metric near the node on each component is flat.

\

\subsection{Type I boundary nodes}

Let $\Sigma=\{(z,w)\in U|z^2-w^2=0\}$. Then $\Sigma$ has two components $\Sigma_1=\{(z,w)\in U|z=w\}$ and $\Sigma_2=\{(z,w)\in U|z=-w\}$, which are attached at the point $\Sigma_1\cap \Sigma_2=\{(0,0)\}$. The map $\si$ gives an involution ($\si^2=id$) on each component $\Sigma_i$. The fixed point set of $\si$ is $F=\{(z,w)\in \Complex^2| Im(z)=Im(w)=0\}$. The boundary of $\Sigma_1/\si$ is just the fixed point set of the involution in $\Sigma_1$, i.e. $\ga_1=F\cap \Sigma_1=\{z=w\in \Real^1\}$, which is a 1-dimensional curve. Similarly, the boundary of $\Sigma_2/\si$ is the curve $\ga_2=F\cap \Sigma_2=\{z=-w\in \Real^1\}$. Thus the boundary curves $\ga_1$ and $\ga_2$ intersects at the node $(0,0)$.

By sending $(z,w)$ to $(z-w,z+w)$, it is clear that the surface $\Si$ here is isomorphic to the one in the interior node case. Indeed, $\Si$ is the complex double of $\Si/\si$. Thus by using the involution map $\si$, the picture of Type 1 boundary node is actually a (\emph{vertical}) half of the interior node case.

\

\subsection{Type II boundary nodes}

Let $\Sigma=\{(z,w)\in U|z^2+w^2=0\}$. Then $\Sigma$ has two components $\Sigma_1=\{(z,w)\in U|z=iw\}$ and $\Sigma_2=\{(z,w)\in U|z=-iw\}$, which are attached at the point $\Sigma_1\cap \Sigma_2=\{(0,0)\}$. The map $\si$ gives an anti-holomorphic ivolution between $\Sigma_1$ and $\Sigma_2$. The only fixed point of $\si$ on $\Si$ is the node $(0,0)$. Thus the boundary of $\Si/\si$ is simply the node itself.

By sending $(z,w)$ to $(z-iw,z+iw)$, it is clear that the surface $\Si$ here is also isomorphic to the one in the interior node case. Indeed, $\Si$ is the complex double of $\Si/\si$. Thus by using the symmetry map $\si$, the picture of Type 2 boundary node is actually a (\emph{horizontal}) half of the interior node case.

\

\subsection{Hyperbolic geometry picture near nodes}

Suppose $\Si$ is a bordered surface of general type $(g,k)$. Let $(\t{\Si}, \pi, \si)$ be the corresponding complex double which is a closed surface with genus $\t{g}=2g + k - 1>1$. Then by uniformization theorem, for each complex structure $c$ on $\Si$, there exists a hyperbolic metric $h$ on $\Si$. It is easy to see that $h$ and $c$ can be extended to its complex double $\t{\Si}$ such that $\t{h}$ and $\t{c}$ is symmetric w.r.t. the (anti-holomorphic) involution $\si$.

Now suppose $\Si_n=(\Si, h_n, c_n)$ is a sequence of degenerating hyperbolic surfaces which converges to a hyperbolic surface $\Si_\infty=(\Si_\infty, h, c)$, with finitely many nodes $\{y_1,\cdots, y_m; z_1,\cdots, z_l\}$ where $y_i$'s are on the boundary and $z_j$'s are interior nodes. Then the complex double $(\t{\Si},\t{h},\t{c})$ converges to the complex double $\t{\Si}_\infty$ of $\Si_\infty$, with nodes $\{y_1,\cdots, y_m; z_1,\cdots, z_l, z_1', \cdots, z_l'\}$ where $z_j'=\si(z_j)\in \Si_\infty'$.  By classical results, $\t{\Si}_\infty$ is obtained by pinching $m+2l$ pairwise disconnected Jordan curves to the nodes.

For simplicity, we first assume there is only one node $z$. If $z$ is an interior nodes, i.e. $z\notin \p\Si_\infty=\Si_\infty^\si$, then there exists a collar area $A_n\subset \Si_n$ near a closed geodesic $\ga_n\subset \Si_n$, which is isomorphic to a hyperbolic cylinder $P_n=[-T_n, T_n]\times S^1$ with $T_n\to\infty$ with metric (cf. \cite{Zhu, RTZ}) $$ds^2=\left(\frac{l}{2\pi\cos(\frac{lt}{2\pi})}\right)^2(dt^2+d\theta^2).$$  Moreover, $\Si_n\setminus A_n$ converges to $\Si_\infty\setminus\{z\}$ smoothly. The local geometry near the node $z$ on each component of $\Si_\infty$ is a standard hyperbolic cusp with metric $(\frac{|dz|}{|z|\ln|z|})^2$.

If $y\in \p\Si_\infty$ is a boundary node, then there also exists a collar area $\t{A}_n\subset \t{\Si}_n$ which lies in the complex double, near a closed geodesic $\t{\ga}_n\subset\t{\Si}_n$ and isomorphic to a hyperbolic cylinder $\t{P}_n=[-T_n, T_n]\times S^1$. Note that the cylinder $\t{P}_n$ is the $\de$-thin part of $\t{\Si}_n$, and hence is symmetric w.r.t. the involution $\si$, i.e. $\si^*ds^2=ds^2$. However, there are only two possible involutions on the hyperbolic cylinder $\t{P}_n$ which is anti-holomorphic as follows. (The antipodal map $\th\to \pi+\th$ is holomorphic.)

The first one corresponds to the symmetry of $\t{P}_n$ w.r.t. the horizontal lines
$$ \xi_n=  \{(t,\th) \ | \ \th=0,\pi \}\subset\t{P}_n. $$
Namely, the involution $\si:\t{P}_n\to \t{P}_n$ maps $(t,\th)$ to $(t, 2\pi-\th)$. In this case, $\xi_n$ lies in the boundary $\p\Si_n$. Let $Q^+_n=[-T_n,T_n]\times [0,\pi]$ be the vertical half of $\t{P}_n$ in $\Si_n$. Then $\Si_n\setminus Q^+_n$ converges to $\Si_\infty\setminus\{y\}$ smoothly. In hyperbolic geometry, the node is obtained by shrinking a geodesic $\ga_n^+=\{0\}\times[0,\pi]$ which connects two points at the boundary. This is exactly the type I node described in the algebraic language above.

The second one corresponds to the symmetry of $\t{P}_n$ w.r.t. the vertical middle circle
$$\t{\ga}_n=  \{(t,\th)\ | \ t=0 \}\subset\t{P}_n. $$
Namely, the involution $\si:\t{P}_n\to \t{P}_n$ maps $(t,\th)$ to $(-t, \th)$. In this case, $\t{\ga}_n$ is fixed by $\si$ and hence belongs to the boundary $\p\Si_n$. Let $P_n^+=[0,T_n]\times S_1$ be the horizontal half of $\t{P}_n$ in $\Si_n$. Then $\Si_n\setminus P_n^+$ converges to $\Si_\infty\setminus\{y\}$ smoothly. In hyperbolic geometry, the node is obtained by shrinking a boundary curve $\ga_n=\t\ga_n$ (also a closed geodesic). This is exactly the type II node described in the algebraic language.

\

\section{Harmonic maps from half cylinders and proof of Theorem \ref{thm:main}}\label{sec:proof-of-thm}

\

In this section, we develop the blow-up analysis for harmonic maps from long half cylinders and then complete the proof of Theorem \ref{thm:main}.

Let $u_n$ be a sequence of harmonic maps defined on the vertical half cylinder
$$Q_n^+=[-T_n,T_n]\times [0,\pi]$$
with free boundary $u_n([-T_n,T_n]\times \{0,\pi\})$ on $K$ and with uniformly bounded energy
$$E(u_n,Q_n^+)\leq \Lambda$$
or defined on the horizontal half cylinder
$$P^+_n=[0,T_n]\times S^1$$
with free boundary $u_n(\{0\}\times S^1)$ on $K$ and with uniformly bounded energy
$$E(u_n,P^+_n)\leq \Lambda,$$
where $$[-T_n,T_n]\times \{0,\pi\}=[-T_n,T_n]\times \{0\}\cup [-T_n,T_n]\times \{\pi\}.$$

\

For a harmonic map $u$ from a cylinder $[-T,T]\times S^1$, using the fact that the Hopf differential $\phi(u)$ is holomorphic, it was observed in \cite{Zhu} that the integration $$\int_{\{t\}\times S^1}\phi(u)d\theta$$ is independent of $t$, which is a complex constant. Here, for harmonic maps from half cylinders, by integrating by parts, we can also define a quantity independent of $t$, which plays an important role to characterize the asymptotic properties of the maps in the limit process.

\begin{lem}\label{lem:01}
If $u$ is a harmonic map defined on $Q_n^+$ with free boundary $u([-T_n,T_n]\times \{0,\pi\})$ on $K$, then $$\int_{\{t\}\times [0,\pi]}\left(|\partial_t u|^2-|\partial_\theta u|^2 \right)d\theta$$ is independent of $t$, where $\partial_t u=\frac{\partial u}{\partial t}$ and $\partial_\theta u=\frac{\partial u}{\partial \theta}$.

If $u$ is a harmonic map defined on $P^+_n$ with free boundary $u_n(\{0\}\times S^1)$ on $K$, then
$$\int_{\{t\}\times S^1}\left(|\partial_t u|^2-|\partial_\theta u|^2 \right)d\theta$$ is independent of $t$.
\end{lem}

\begin{proof}
We shall only prove for first case, since the second case is similar and easier.

Since $u$ is a harmonic map defined on $Q_n$, denoting $Q_{t_1t_2}=[t_1,t_2]\times [0,\pi]$, then we have
\begin{align*}
0&=\int_{Q_{t_1t_2}}\Delta u\cdot \partial_t ud\theta dt\\
&=\int_{Q_{t_1t_2}}(\partial^2_t u+\partial^2_\theta u)\cdot \partial_t ud\theta dt\\
&=\frac{1}{2}\int_{Q_{t_1t_2}} \frac{\partial}{\partial t}(|\partial_tu|^2-|\partial_\theta u|^2) d\theta dt+\int_{t_1}^{t_2}\partial_t u\cdot\partial_\theta u|_0^{\pi}dt\\
&=\frac{1}{2}\left(\int_{\{t_2\}\times [0,\pi]} (|\partial_tu|^2-|\partial_\theta u|^2) d\theta- \int_{\{t_1\}\times [0,\pi]} (|\partial_tu|^2-|\partial_\theta u|^2) d\theta  \right),
\end{align*}
where we used the free boundary condition, i.e. $$\partial_t u\cdot \partial_\theta u|_{[-T_n, T_n]\times \{0,\pi\}}=0.$$ Then the conclusions of lemma follow immediately.
\end{proof}

\

By above lemma, we shall give the following definition of Pohozaev type constant.
\begin{defn}
The constant
\begin{equation}\label{equat:08}
\alpha_n:=\int_{\{t\}\times [0,\pi]}\left (|\partial_t u_n|^2-|\partial_\theta u_n|^2 \right)d\theta\quad  \text{~for~}  Q_n^+\end{equation} or
\begin{equation}\label{equat:07}
\alpha_n:=\int_{\{t\}\times S^1}\left (|\partial_t u_n|^2-|\partial_\theta u_n|^2 \right )d\theta\quad  \text{~for~}  P_n^+\end{equation} is called the Pohozaev type constant for $u_n$ on $Q^+_n$ or $P_n^+$, respectively.
\end{defn}

\

In this paper, we want to study the energy concentration of $u_n$ on $Q_n^+$ and $P^+_n$. Firstly, by the classical blow-up theory of harmonic maps near an interior energy concentration point \cite{DingWeiyueandTiangang}, near a (free) boundary energy concentration point \cite{jost-Liu-Zhu} and near the interior degenerating area \cite{Zhu}, it is not hard to prove the following:

\begin{thm}\label{thm:03}
$u_n:P^+_n\to N$ be a sequence of harmonic maps with free boundary $u_n(\{0\}\times S^1)$ on $K$ and with uniformly bounded energy $$E(u_n; P_n^+)\leq\Lambda,$$ where $P_n^+$ is a cylinder with standard flat metric $ds^2=dt^2+d\theta^2$ and $T_n\to\infty$ as $n\to\infty$.

Then there exist a finite harmonic spheres $v^i:S^2\to N,\ i=1,...,I$ and harmonic disks $w^j:D_1(0)\to N,\ j=1,...,J$ with free boundary $w^j(\partial D_1(0))$ on $K$ such that, after passing to a subsequence, there holds
\begin{align*}
\lim_{n\to\infty} E(u_n,P^+_n)=\sum_{i=1}^IE(v^i)+\sum_{j=1}^JE(w^j)+\lim_{n\to\infty}2\alpha_nT_n,
\end{align*}
where $\alpha_n$ is defined by \eqref{equat:07}.

\end{thm}
\begin{proof}
With the help of \cite{DingWeiyueandTiangang} and \cite{jost-Liu-Zhu}, one can refer to a similar proof in \cite{Zhu} or \cite{Chen-Li-Wang}. We omit the details here.
\end{proof}

\

Next, we will focus on the case of $Q_n^+$. We first consider a simpler case by assuming that there is no energy concentration points in $Q_n^+$, i.e.
\begin{equation}\label{equat:01}
\lim_{n\to\infty}\sup_{-T_n\leq t\leq T_n-1}E(u_n, [t,t+1]\times [0,\pi])=0.
\end{equation}

\

By the small energy regularity theory of harmonic maps in the interior case \cite{SU} and the free boundary case \cite{jost-Liu-Zhu}, we have
\begin{lem}\label{lem:small-energy-regularity}
Let $u_n:Q_n^+\to N$ be a sequence of harmonic maps with free boundary $u_n([-T_n,T_n]\times \{0,\pi\})$ on $K$. Assuming \eqref{equat:01} holds, then for any $\epsilon>0$, there hold
\begin{equation}
\sup_{-T_n+1\leq t\leq T_n-1}\|\nabla u_n\|_{W^{1,2}([t-\frac{1}{2},t+\frac{1}{2}]\times [0,\pi])}\leq C(K,N)\sup_{-T_n+1\leq t\leq T_n-1}\|\nabla u_n\|_{L^2([t-1,t+1]\times [0,\pi])}
\end{equation}
and
\begin{equation}
\sup_{-T_n+1\leq t\leq T_n-1}\|u_n\|_{Osc([t-\frac{1}{2},t+\frac{1}{2}]\times [0,\pi])}\leq C(K,N)\sup_{-T_n+1\leq t\leq T_n-1}\|\nabla u_n\|_{L^2([t-1,t+1]\times [0,\pi])}\leq C(K,N)\epsilon
\end{equation}
when $n$ is big enough, where $$\|u_n\|_{Osc(\Omega)}:=\sup_{x,y\in \Omega}|u_n(x)-u_n(y)|.$$

In particular, the free boundary condition tells us that the image of $u_n$ is contained in a small tubular neighborhood of $K$ in $N$, i.e. $$u_n([-T_n+1,T_n-1]\times [0,\pi])\subset K_{C(N)\epsilon},$$ where $K_{C(N)\epsilon}$ denotes the $C(N)\epsilon$-tubular neighborhood of $K$ in $N$.
\end{lem}

\

Denote by $K_{\delta_0}$ the $\delta_0$-tubular neighborhood of $K$ in $N$. Taking $\delta_0>0$ small enough, then for any $y\in K_{\delta_0}$,
there exists a unique projection $y'\in K$. Set $\overline{y}=exp_{y'}\{-exp^{-1}_{y'}y\}$. So we may define an involution $\sigma$, $i.e.$ $\sigma^2=Id$ as in
\cite{Hamilton, GJ, Scheven} by
\[
\sigma(y)=\overline{y} \quad for \quad y\in K_{\delta_0}.
\]
Then it is easy to check that the linear operator $D\sigma:TN|_{K_{\delta_0}}\to TN|_{K_{\delta_0}}$ satisfies $D\sigma(V)=V$ for $V\in TK$ and $D\sigma(\xi)=-\xi$ for $\xi\in T^\perp K$.

By Lemma \ref{lem:small-energy-regularity}, we can define an extension of $u_n$ to $Q_n:=[-T_n+1,T_n-1]\times S^1$ that
\begin{align}\label{def:function}
\widehat{u}_n(t,\theta)=
\begin{cases}
u_n(t,\theta),\quad &if \quad (t,\theta)\in [-T_n+1,T_n-1]\times [0,\pi];\\
\sigma(u_n(\rho(t,\theta))) ,\quad &if \quad (t,\theta)\in [-T_n+1,T_n-1]\times [\pi,2\pi],
\end{cases}
\end{align}
where $\rho(t,\theta):=(t,2\pi-\theta)$.

\

Now, we derive the equation for the extended map $\widehat{u}_n$. One can see that it is no longer a harmonic map, but it satisfies the following property.
\begin{prop}
Let $u\in W^{2,p}([T_1,T_2]\times [0,\pi],N)$, $1\leq p\leq \infty$, be a map with free boundary $u([T_1,T_2]\times \{0,\pi\})$ on $K$. Let $u([T_1,T_2]\times [0,\pi])\subset K_{\delta_0}$, then the extended map $\widehat{u}$ defined by \ref{def:function} satisfying $\widehat{u}\in W^{2,p}([T_1,T_2]\times S^1)$ and
\begin{equation}\label{equat:09}
\Delta \widehat{u}+\Upsilon_{\widehat{u}}(\nabla\widehat{u},\nabla\widehat{u})=0\quad in \quad [T_1,T_2]\times S^1,
\end{equation}
where $\Upsilon_{\widehat{u}}(\cdot,\cdot)$ is a bounded bilinear form defined by
\begin{align*}
\Upsilon_{\widehat{u}}(\cdot,\cdot)=
\begin{cases}
A(\widehat{u})(\cdot,\cdot)\ in\ [T_1,T_2]\times (0,\pi),\\
A(\widehat{u})(\cdot,\cdot)-D^2\sigma|_{\sigma(\widehat{u})}(D\sigma|_{\widehat{u}}\circ\cdot\ ,
D\sigma|_{\widehat{u}}\circ\cdot)\ in\ [T_1,T_2]\times (\pi,2\pi);
\end{cases}\end{align*} satisfying
\[
|\Upsilon_{\widehat{u}}(\nabla\widehat{u},\nabla\widehat{u})|\leq C(K,N)|\nabla\widehat{u}|^2.
\]
\end{prop}
\begin{proof}
According to the properties of $D\sigma$, it is easy to see that $\widehat{u}\in W^{2,p}([T_1,T_2]\times S^1)$  since $u\in W^{2,p}([T_1,T_2]\times [0,\pi],N)$ and satisfies free boundary condition. Next, we derive the equation for $\widehat{u}$.

Firstly, for $(t,\theta)\in (T_1,T_2)\times (\pi,2\pi)$, we have
\begin{align*}
\Delta\widehat{u}+A(\widehat{u})(\nabla\widehat{u},\nabla\widehat{u})=\nabla^{\widehat{u}^*TN}_{d\widehat{u}(e_\beta)}(d\widehat{u}(e_\beta)),
\end{align*}
where $\nabla^{\widehat{u}^*TN}$ is the covariant derivative on the pull back bundle.

Computing directly in local normal coordinates $\{\frac{\partial}{\partial y^i}\}_{i=1}^n$ of target manifold $N$, we get
\begin{align*}
\nabla^{\widehat{u}^*TN}_{d\widehat{u}(e_\beta)}(d\widehat{u}(e_\beta))&=\nabla^{\widehat{u}^*TN}_{d\widehat{u}(e_\beta)}\left(\frac{\partial\sigma^j}{\partial y^i}d u^i\circ d\rho(e_\beta)\frac{\partial}{\partial y^j}\right)\\
&=\frac{\partial^2\sigma^j}{\partial y^k\partial y^i}d u^i\circ d\rho(e_\beta)d u^k\circ d\rho(e_\beta)\frac{\partial}{\partial y^j}+\frac{\partial\sigma^j}{\partial y^i}\Delta u^i\frac{\partial}{\partial y^j}\\
&=\frac{\partial^2\sigma^j}{\partial y^k\partial y^i}d u^i\circ d\rho(e_\beta)d u^k\circ d\rho(e_\beta)\frac{\partial}{\partial y^j} \\& =D^2\sigma|_{\sigma(\widehat{u})}(D\sigma|_{\widehat{u}}\circ\nabla\widehat{u}, D\sigma|_{\widehat{u}}\circ\nabla\widehat{u}),
\end{align*}
where the last second equality follows from the fact that $u$ is a harmonic map which is equivalent to say that $\Delta u^i=0,\ i=1,...,n$ in local normal coordinate system $\{\frac{\partial}{\partial y^i}\}_{i=1}^n$.

Combining this with the fact that $\widehat{u}$ is a harmonic map in $[T_1,T_2]\times [0,\pi]$, the conclusions of the proposition follows immediately.
\end{proof}

\

Now, we estimate the energy of $u_n$ on $Q_n^+$ when there is no energy concentration. We have

\begin{thm}\label{thm:01}
Let $N$ be a compact Riemannian manifold and $K\subset N$ is a smooth submanifold. Let $u_n:Q_n^+\to N$ be a sequence of harmonic maps with free boundary $u_n([-T_n,T_n]\times\{0,\pi\})$ on $K$ and  with uniformly bounded energy $$E(u_n; Q_n^+)\leq\Lambda,$$ where $Q_n^+$ is a cylinder with standard flat metric $ds^2=dt^2+d\theta^2$ and $T_n\to\infty$ as $n\to\infty$.

Suppose there is no energy concentration for $u_n$, i.e. \eqref{equat:01} holds, then we have
\begin{align*}
\lim_{n\to\infty} E(u_n; Q_n^+)=\lim_{n\to\infty}2 T_n\alpha_n,
\end{align*}
where $\alpha_n$ is defined by \eqref{equat:08}.
\end{thm}

\begin{proof}
Since \eqref{equat:01} holds, by Lemma \ref{lem:small-energy-regularity}, we know that for any $\epsilon>0$, there holds $$u_n([-T_n+1,T_n-1]\times [0,\pi])\subset K_{C(N)\epsilon},$$  when $n$ is big enough. Taking $\epsilon>0$ small such that $C(N)\epsilon\leq \delta_0$, then we can use the definition \eqref{def:function} to extend $u_n$ to $\widehat{u}_n$, which is defined on $[-T_n+1,T_n-1]\times S^1$ and satisfies equation \eqref{equat:09}.

Setting $$\widehat{u}^*_n(t):=\frac{1}{2\pi}\int_0^{2\pi}\widehat{u}_n(t,\theta)d\theta,\ t\in [-T_n+1,T_n-1],$$ then, by \eqref{equat:01} and Lemma \ref{lem:small-energy-regularity}, we have
\begin{align}\label{equat:10}
\|\widehat{u}_n-\widehat{u}_n^*\|_{L^\infty([-T_n+1,T_n-1]\times S^1)}&\leq \sup_{t\in [-T_n+1,T_n-1]}\|\widehat{u}_n\|_{Osc(\{t\}\times S^1)}\notag\\
&\leq C(K,N)\sup_{t\in [-T_n+1,T_n-1]}\|u_n\|_{Osc(\{t\}\times [0,\pi])}\notag\\&\leq C(K,N)\epsilon,
\end{align}
when $n$ is big enough.

Multiplying the equation \eqref{equat:09} by $\widehat{u}_n-\widehat{u}_n^*$ and integrating by parts, we get
\begin{align}\label{equat:02}
&\int_{[-T_n+1,T_n-1]\times S^1}\Upsilon_{\widehat{u}_n}(\nabla \widehat{u}_n,\nabla \widehat{u}_n)(\widehat{u}_n-\widehat{u}_n^*)dtd\theta\notag\\&=\int_{[-T_n+1,T_n-1]\times S^1}-\Delta \widehat{u}_n (\widehat{u}_n-\widehat{u_n}^*) dtd\theta\notag\\
&=\int_{[-T_n+1,T_n-1]\times S^1}\nabla \widehat{u}_n\nabla (\widehat{u}_n-\widehat{u}_n^*)dtd\theta-\int_{\{T_n-1\}\times S^1}\frac{\partial \widehat{u}_n}{\partial t}(\widehat{u}_n-\overline{u}_n^*)d\theta\notag\\&\quad+\int_{\{-T_n+1\}\times S^1}\frac{\partial \widehat{u}_n}{\partial t}(\widehat{u}_n-\overline{u}_n^*)d\theta.
\end{align}
Using H\"{o}lder's inequality, we have
\begin{align}\label{equat:03}
&\int_{[-T_n+1,T_n-1]\times S^1}\nabla \widehat{u}_n\nabla (\widehat{u}_n-\widehat{u}_n^*)dtd\theta\notag\\&=\int_{[-T_n+1,T_n-1]\times S^1}\left(|\nabla \widehat{u}_n|^2-\frac{\partial \widehat{u}_n}{\partial t}\frac{\partial \widehat{u}^*_n}{\partial t}\right)dtd\theta\notag\\
&\geq \int_{[-T_n+1,T_n-1]\times S^1}|\nabla \widehat{u}_n|^2dtd\theta-\left(\int_{[-T_n+1,T_n-1]\times S^1}|\frac{\partial \widehat{u}_n}{\partial t}|^2dtd\theta\right)^{\frac{1}{2}}\left(\int_{[-T_n+1,T_n-1]\times S^1}|\frac{\partial \widehat{u}_n^*}{\partial t}|^2dtd\theta\right)^{\frac{1}{2}}\notag\\ &\geq \int_{[-T_n+1,T_n-1]\times S^1}\left(|\nabla \widehat{u}_n|^2-|\frac{\partial \widehat{u}_n}{\partial t}|^2\right)dtd\theta,
\end{align}
where the last inequality follows from the fact that
\begin{align*}
\left(\int_{[-T_n+1,T_n-1]\times S^1}|\frac{\partial \widehat{u}_n^*}{\partial t}|^2dtd\theta\right)^{\frac{1}{2}}&= \left(\int_{[-T_n+1,T_n-1]\times S^1}\left|\frac{1}{2\pi}\int_0^{2\pi}\frac{\partial \widehat{u}_n}{\partial t}d\theta\right|^2dtd\theta\right)^{\frac{1}{2}}\\
&\leq \left(\int_{[-T_n+1,T_n-1]\times S^1}\frac{1}{2\pi}\int_0^{2\pi}|\frac{\partial \widehat{u}_n}{\partial t}|^2d\theta dtd\theta\right)^{\frac{1}{2}}\\
&=\left(\int_{[-T_n+1,T_n-1]\times S^1}|\frac{\partial \widehat{u}_n}{\partial t}|^2dtd\theta\right)^{\frac{1}{2}}.
\end{align*}

Combining \eqref{equat:02} and \eqref{equat:03}, we arrived
\begin{align}\label{equat:04}
&\int_{[-T_n+1,T_n-1]\times S^1}\left(|\nabla \widehat{u}_n|^2-|\frac{\partial \widehat{u}_n}{\partial t}|^2\right)dtd\theta\notag\\ &\leq \int_{[-T_n+1,T_n-1]\times S^1}\Gamma_{\widehat{u}_n}(\nabla \widehat{u}_n,\nabla \widehat{u}_n)(\widehat{u}_n-\widehat{u}_n^*)dtd\theta\notag\\&\quad+\int_{\{T_n-1\}\times S^1}\frac{\partial \widehat{u}_n}{\partial t}(\widehat{u}_n-\overline{u}_n^*)d\theta-\int_{\{-T_n+1\}\times S^1}\frac{\partial \widehat{u}_n}{\partial t}(\widehat{u}_n-\overline{u}_n^*)d\theta.
\end{align}

A direct computation yields that
\begin{align}\label{equat:05}
\int_{-T_n+1}^{T_n-1}\int_\pi^{2\pi}\left(|\nabla \widehat{u}_n|^2-|\frac{\partial \widehat{u}_n}{\partial t}|^2\right)dtd\theta =\int_{-T_n+1}^{T_n-1}\int^\pi_{0}\left(|D\sigma(u_n)\cdot\nabla u_n|^2-|D\sigma(u_n)\cdot\frac{\partial u_n}{\partial t}|^2\right)dtd\theta.
\end{align}
Note that
\begin{align*}
|D\sigma(u_n)\cdot\nabla u_n|^2&=\langle D\sigma\cdot\nabla u_n,D\sigma\cdot\nabla u_n\rangle\\&=\langle (D\sigma)^*\cdot D\sigma\cdot\nabla u_n,\nabla u_n\rangle\\&= \left\langle \bigg((D\sigma)^*\cdot D\sigma-Id\bigg)\cdot\nabla u_n,\nabla u_n\right\rangle+|\nabla u_n|^2,
\end{align*}
where $(D\sigma)^*$ is the adjoint operator of linear operator $D\sigma:TN|_{K_{\delta_0}}\to TN|_{K_{\delta_0}}$.

Similarly, $$|D\sigma\cdot\frac{\partial u_n}{\partial t}|^2=\left\langle \bigg((D\sigma)^*\cdot D\sigma-Id\bigg)\cdot\frac{\partial u_n}{\partial t},\frac{\partial u_n}{\partial t}\right\rangle+|\frac{\partial u_n}{\partial t}|^2.$$

Noting that $(D\sigma)^*\cdot D\sigma|_K=Id$, by the continuity of eigenvalues of $(D\sigma)^*\cdot D\sigma$, we have that for any $\delta'>0$, there exists a constant  $\delta_1=\delta_1(\delta',K,N)$, such that for any $y\in K_{\delta_1}$ and $\xi\in TN|_{K_{\delta_1}}$, there holds $$\langle (D\sigma)^*|_{y}\cdot D\sigma|_y\cdot\xi, \xi\rangle\leq (1+\delta')|\xi|^2.$$

By Lemma \ref{lem:small-energy-regularity}, for any $\epsilon>0$, when $n$ is big enough, there holds $$\|dist(u_n,K)\|_{L^\infty([-T_n+1,T_n-1]\times S^1)}\leq C\epsilon.$$ Thus, $$\left\langle \bigg((D\sigma)^*|_{u_n(t,\theta)}\cdot D\sigma|_{u_n(t,\theta)}-Id\bigg)\cdot\xi, \xi\right\rangle\leq C\epsilon|\xi|^2,\ (t,\theta)\in [-T_n+1,T_n-1]\times S^1,$$ which implies that
\begin{align*}
\int_{-T_n+1}^{T_n-1}\int_\pi^{2\pi}\left(|\nabla \widehat{u}_n|^2-|\frac{\partial \widehat{u}_n}{\partial t}|^2\right)dtd\theta\geq \int_{-T_n+1}^{T_n-1}\int^{\pi}_{0}\left(|\nabla u_n|^2-|\frac{\partial u_n}{\partial t}|^2\right)dtd\theta-C\epsilon \int_{-T_n+1}^{T_n-1}\int^{\pi}_{0}|\nabla u_n|^2dtd\theta.
\end{align*}

Therefore, we have
\begin{align*}
&\int_{[-T_n+1,T_n-1]\times S^1}\left(|\nabla \widehat{u}_n|^2-|\frac{\partial \widehat{u}_n}{\partial t}|^2\right)dtd\theta\\&=\int_{-T_n+1}^{T_n-1}\int^{\pi}_{0}\left(|\nabla u_n|^2-|\frac{\partial u_n}{\partial t}|^2\right)dtd\theta +\int_{-T_n+1}^{T_n-1}\int_\pi^{2\pi}\left(|\nabla \widehat{u}_n|^2-|\frac{\partial \widehat{u}_n}{\partial t}|^2\right)dtd\theta \\&\geq 2\int_{-T_n+1}^{T_n-1}\int^{\pi}_{0}\left(|\nabla u_n|^2-|\frac{\partial u_n}{\partial t}|^2\right)dtd\theta-C\epsilon \int_{-T_n+1}^{T_n-1}\int^{\pi}_{0}|\nabla u_n|^2dtd\theta.
\end{align*}
Combining this with \eqref{equat:04}, \eqref{equat:10} and Lemma \ref{lem:small-energy-regularity}, it yields
\begin{align}\label{equat:06}
&2\int_{-T_n+1}^{T_n-1}\int^{\pi}_{0}\left(|\nabla u_n|^2-|\frac{\partial u_n}{\partial t}|^2\right)dtd\theta\notag\\ &\leq \int_{[-T_n+1,T_n-1]\times S^1}\Gamma_{\widehat{u}_n}(\nabla \widehat{u}_n,\nabla \widehat{u}_n)(\widehat{u}_n-\widehat{u}_n^*)dtd\theta+C\epsilon \int_{-T_n+1}^{T_n-1}\int^{\pi}_{0}|\nabla u_n|^2dtd\theta\notag\\&\quad+\int_{\{T_n-1\}\times S^1}\frac{\partial \widehat{u}_n}{\partial t}(\widehat{u}_n-\widehat{u}_n^*)d\theta-\int_{\{-T_n+1\}\times S^1}\frac{\partial \widehat{u}_n}{\partial t}(\widehat{u}_n-\widehat{u}_n^*)d\theta\notag\\ &\leq C\epsilon \int_{-T_n+1}^{T_n-1}\int^{\pi}_{0}|\nabla u_n|^2dtd\theta+\int_{\{T_n-1\}\times S^1}\frac{\partial \widehat{u}_n}{\partial t}(\widehat{u}_n-\widehat{u}_n^*)d\theta-\int_{\{-T_n+1\}\times S^1}\frac{\partial \widehat{u}_n}{\partial t}(\widehat{u}_n-\widehat{u}_n^*)d\theta.
\end{align}

For the boundary terms, by trace theory and Lemma \ref{lem:small-energy-regularity}, we have
\begin{align*}
\left|\int_{\{T_n-1\}\times S^1}\frac{\partial \widehat{u}_n}{\partial t}(\widehat{u}_n-\widehat{u}_n^*)d\theta\right|&\leq \|\widehat{u}_n-\widehat{u}_n^*\|_{L^\infty(\{T_n-1\}\times S^1)}\int_{\{T_n-1\}\times S^1}|\frac{\partial \widehat{u}_n}{\partial t}|d\theta\\
&\leq C\|\widehat{u}_n\|_{Osc(\{T_n-1\}\times S^1)}\int_{\{T_n-1\}\times [0,\pi]}|\frac{\partial u_n}{\partial t}|d\theta\\
&\leq C\|u_n\|_{Osc(\{T_n-1\}\times [0,\pi])}\left(\|\nabla^2 u_n\|_{L^2([T_n-\frac{3}{2},T_n-\frac{1}{2}]\times [0,\pi])}+\|\nabla u_n\|_{L^2([T_n-\frac{3}{2},T_n-\frac{1}{2}]\times [0,\pi])}\right)\\
&\leq C\|u_n\|_{Osc(\{T_n-1\}\times [0,\pi])}\|\nabla u_n\|_{L^2([T_n-2,T_n]\times [0,\pi])}\\&\leq C\|\nabla u_n\|^2_{L^2([T_n-2,T_n]\times [0,\pi])}\leq C\epsilon,
\end{align*}
where the last inequality follows from \eqref{equat:01}.

Similarly,
\[
\left|\int_{\{-T_n+1\}\times S^1}\frac{\partial \widehat{u}_n}{\partial t}(\widehat{u}_n-\widehat{u}_n^*)d\theta\right|\leq C\|\nabla u_n\|^2_{L^2([-T_n,-T_n+2]\times [0,\pi])}\leq C\epsilon.
\]

Then by \eqref{equat:06} and Lemma \ref{lem:01}, we have
\begin{align*}
C\epsilon&\geq 2\int_{-T_n+1}^{T_n-1}\int^{\pi}_{0}\left(|\nabla u_n|^2-|\frac{\partial u_n}{\partial t}|^2\right)dtd\theta\\
&=\left|\int_{-T_n+1}^{T_n-1}\int^{\pi}_{0}|\nabla u_n|^2dtd\theta-\int_{-T_n+1}^{T_n-1}\int^{\pi}_{0}\left(|\frac{\partial u_n}{\partial t}|^2-|\frac{\partial u_n}{\partial \theta}|^2\right)dtd\theta\right|\\
&=\left|\int_{-T_n+1}^{T_n-1}\int^{\pi}_{0}|\nabla u_n|^2dtd\theta-2 (T_n-1)\alpha_n\right|,
\end{align*}
which implies
\begin{align*}
\lim_{n\to\infty} E(u_n; [-T_n+1,T_n-1]\times [0,\pi])=\lim_{n\to\infty}2 T_n\alpha_n.
\end{align*}
We finished the proof of theorem.
\end{proof}

\

Next, we will consider the more general case of allowing the phenomenon of energy concentration, i.e. \eqref{equat:01} does not hold. Combining Theorem  \ref{thm:01} with \cite{DingWeiyueandTiangang,jost-Liu-Zhu,Zhu}, we can prove the following theorem.
\begin{thm}\label{thm:02}
Let $N$ be a compact Riemannian manifold and $K\subset N$ is a smooth submanifold. Let $u_n:Q_n^+\to N$ be a sequence of harmonic maps with free boundary $u_n([-T_n,T_n]\times\{0,\pi\})$ on $K$ and  with uniformly bounded energy $$E(u_n; Q_n^+)\leq\Lambda,$$ where $Q_n^+$ is a cylinder with standard flat metric $ds^2=dt^2+d\theta^2$ and $T_n\to\infty$ as $n\to\infty$.

Then there exist a finite harmonic spheres $v^i:S^2\to N,\ i=1,...,I$ and harmonic disks $w^j:D_1(0)\to N,\ j=1,...,J$ with free boundary $w^j(\partial D_1(0))$ on $K$ such that, after passing to a subsequence, there holds
\begin{align*}
\lim_{n\to\infty} E(u_n,Q^+_n)=\sum_{i=1}^IE(v^i)+\sum_{j=1}^JE(w^j)+\lim_{n\to\infty}2 \alpha_nT_n,
\end{align*}
where $\alpha_n$ is defined by \eqref{equat:08}.

\end{thm}

\begin{proof}
Based on the neck analysis scheme in \cite{DingWeiyueandTiangang}, Zhu \cite{Zhu} gives a refined ``bubble domain and neck domain" decomposition for a sequence of harmonic maps from a long cylinder. Here, Combining \cite{Zhu} with the blow-up theory for a blow-up sequence of harmonic maps near a (free) boundary point \cite{jost-Liu-Zhu}, it is not hard to see this kind of  decomposition also holds for a sequence of harmonic maps $u_n$ from a long half cylinder $Q_n^+$ with free boundary $u_n([-T_n,T_n]\times \{0,\pi\})$ on $K$. We leave the detailed proof to interested readers.

Precisely, we can show that there exists a constant $L>0$ independent of $n$ and $2L$ sequences $\{a_n^1\}$, $\{b_n^1\}$, $\{a_n^2\}$, $\{b_n^2\}$,..., $\{a_n^L\}$, $\{b_n^L\}$, such that
\[
-T_n\leq a_n^1\ll b_n^1\leq a_n^2\ll b_n^2\leq...\leq a_n^L\ll b_n^L\leq T_n \quad  (a_n^i\ll b_n^i\ means \ \lim_{n\to\infty}b_n^i-a_n^i=\infty)
\]
and $(b_n^i-a_n^i)\ll T_n$, i.e.
\[
\lim_{n\to\infty}\frac{b_n^i-a_n^i}{T_n}=0,\ i=1,...,L.
\]

Denote
\[
J_n^j:=[a_n^j,b_n^j]\times S^1,\ j=1,...,L,
\]
\[
I_n^0:=[-T_n,a_n^1]\times S^1,\ I_n^L:=[b_n^L,T_n]\times S^1,\ I_n^i:=[b_n^i,a_n^{i+1}]\times S^1,\ i=1,...,L-1.
\]
Then, after passing to a subsequence, which we still denote by $\{u_n\}$, there hold

\begin{itemize}
\item[(1)] $\forall i=0,1,...,L$, $$\lim_{n\to\infty}\sup_{t\in I_n^i}\int_{[t,t+1]\times S^1}|\nabla u_n|^2=0.$$ The maps $\phi_n$ on $I_n^i$ are necks corresponding to collapsing homotopically nontrivial curves.

\item[(2)] $\forall j=1,...,L$, there are finitely many harmonic maps $v^{j,l}:S^2\to N$, $l=1,..., l_j$, and finitely harmonic maps $w^{j,m}:D_1(0)\to N$, $m=1,...,m_j$ with free boundary $w^{j,m}(\partial D_1(0))$ on $K$, such that
\begin{align}\label{equat:11}
\lim_{n\to\infty}E(u_n,J_n^j)&=\sum_{l=1}^{l_j}E(v^{j,l})+\sum_{m=1}^{m_j}E(w^{j,m}).
\end{align}
\end{itemize}

By Theorem \ref{thm:01}, we have
\begin{align*}
\lim_{n\to\infty} \sum_{i=0}^LE(u_n;I_n^i)&=\lim_{n\to\infty} \alpha_n\left(a_n^1-(-T_n)+T_n-b_n^L+\sum_{i=1}^{L-1}(a_n^{i+1}-b_n^i)\right)\\
&=\lim_{n\to\infty} \alpha_n\left(2T_n-\sum_{i=1}^{L}(b_n^i-a_n^i)\right)\\
&=\lim_{n\to\infty} 2\alpha_nT_n\left(1-\sum_{i=1}^{L}\frac{b_n^i-a_n^i}{2T_n}\right)=\lim_{n\to\infty} 2\alpha_nT_n.
\end{align*}
Combining this with \eqref{equat:11}, we proved the conclusion of the theorem.
\end{proof}

\

At the end of this section, we give the proof of our main Theorem \ref{thm:main}.
\begin{proof}[\textbf{Proof of Theorem \ref{thm:main}} ]
With the help of the blow-up analysis for a sequence of harmonic maps near an interior blow-up point \cite{DingWeiyueandTiangang}, near a free boundary point \cite{jost-Liu-Zhu}, and the asymptotic behaviour near an interior node \cite{Zhu}, it is easy to see that the conclusions of Theorem \ref{thm:main} follow immediately from Theorem \ref{thm:02} and Theorem \ref{thm:03}.
\end{proof}

\


\end{document}